\documentclass[12pt,reqno]{article}
\usepackage{amssymb,amscd,amsmath,amsthm,color}
\newtheorem{theorem}{Theorem}[section]
\newtheorem{lemma}[theorem]{Lemma}
\newtheorem{cor}[theorem]{Corollary}

\usepackage{enumerate,amssymb}
\theoremstyle{definition}

\newtheorem{example}[theorem]{Example}

\newtheorem{conjecture}[theorem]{Conjecture}
\theoremstyle{remark}
\newtheorem{remark}[theorem]{Remark}

\numberwithin{equation}{section}

\def\bM{\mathbb{M}}

\begin{document}
\baselineskip=15pt

\title{ On a decomposition lemma for  positive semi-definite block-matrices }

\author{Jean-Christophe Bourin{\footnote{Supported by ANR 2011-BS01-008-01.}},  Eun-Young Lee{\footnote{Research  supported by  Basic Science Research Program
through the National Research Foundation of Korea (NRF) funded by the Ministry of Education,
Science and Technology (2010-0003520)}}, Minghua Lin}

\date{January, 2012}

\maketitle

\begin{abstract}
\noindent
This short note, in part of expository nature, points out several new or recent consequences of a  quite nice decomposition for  positive semi-definite matrices.
\end{abstract}

{\small\noindent
Keywords: Operator inequalities, positive semi-definite matrix, unitary orbit,
symmetric norm.
\noindent
AMS subjects classification 2010:  15A60, 47A30,  15A42.}

\section{A decomposition lemma}
\noindent
Let $\bM_n$ denote the space of $n\times n$ complex matrices, or operators on a finite dimensional Hilbert space, and let  $\bM_n^+$ be the positive (semi-definite) part.
For positive  block-matrices,
 $$\begin{bmatrix} A &X \\
X^* &B\end{bmatrix}\in \bM_{n+m}^+, \qquad \hbox{with}~A\in\bM_n^+, \,  B\in\bM_m^+,
$$
we have a remarkable  decomposition lemma  for elements in $\bM_{n+m}^+$ noticed in \cite{BL}:

\begin{lemma} \label{BL-lemma} For every matrix in  $\bM_{n+m}^+$ written in blocks, we have the decomposition
\begin{equation*}
\begin{bmatrix} A &X \\
X^* &B\end{bmatrix} = U
\begin{bmatrix} A &0 \\
0 &0\end{bmatrix} U^* +
V\begin{bmatrix} 0 &0 \\
0 &B\end{bmatrix} V^*
\end{equation*}
for some unitaries $U,\,V\in  \bM_{n+m}$.
\end{lemma}

The motivation for such a decomposition is various inequalities for convex or concave functions of positive operators partitioned in blocks. These results are extensions of some classical majorization, Rotfel'd and Minkowski type inequalities. Lemma \ref{BL-lemma} actually implies a host of such inequalities as shown in the recent papers \cite{BH1} and \cite{BL} where a proof of Lemma \ref{BL-lemma} can be found too.

This note aims to give in the next section further  consequences of the above decomposition lemma, including a simple proof of a recent majorization shown in \cite{LW}.

 Most of the corollaries  below are rather straightforward consequences of Lemma \ref{BL-lemma}, except Corollary \ref{cor-p} which also requires some more elaborated estimates. Corollary \ref{LW-maj} is the majorization given in \cite{LW},  a remarkable extension of the basic and useful inequality
\begin{equation}\label{class}
\left\| \begin{bmatrix} A &0 \\
0 &B\end{bmatrix} \right\| \le
\left\|  A +B \right\|
\end{equation}
for all $A, B\in \bM_n^+$ and all symmetric (or unitarily invariant) norms. The recent survey \cite{H} provides a good exposition on these classical norms.

\section{Some Consequences}

If we first use a unitary congruence with
\begin{equation*}
J=\frac{1}{\sqrt{2}}\begin{bmatrix} I &-I \\
I &I\end{bmatrix}
\end{equation*}
where $I$ is the identity of $\bM_n$, we observe that
\begin{equation*}
J\begin{bmatrix} A &X \\
X^* &B\end{bmatrix} J^*= \begin{bmatrix} \frac{A+B}{2} +{\mathrm{Re}} X& \ast \\
\ast & \frac{A+B}{2} -{\mathrm{Re}} X\end{bmatrix}
\end{equation*}
where $\ast$ stands for unspecified entries and ${\mathrm{Re}}X=(X+X^*)/2$. Thus  Lemma \ref{BL-lemma} yields:

\begin{cor} \label{Lin-prop2} For every matrix in  $\bM_{2n}^+$ written in blocks of the same size, we have a decomposition
\begin{equation*}
\begin{bmatrix} A &X \\
X^* &B\end{bmatrix} = U
\begin{bmatrix} \frac{A+B}{2} +{\mathrm{Re}} X &0 \\
0 &0\end{bmatrix} U^* +
V\begin{bmatrix} 0 &0 \\
0 & \frac{A+B}{2} -{\mathrm{Re}} X\end{bmatrix} V^*
\end{equation*}
for some unitaries $U,\,V\in  \bM_{2n}$.
\end{cor}

 This is equivalent to Corollary \ref{Lin-prop1} below by the obvious unitary congruence
\begin{equation*}
\begin{bmatrix} A &X \\
X^* &B\end{bmatrix}
 \simeq \begin{bmatrix} A &iX \\
-iX^* &B\end{bmatrix}.
\end{equation*}

In Corollary \ref{Lin-prop2}, if $A,B,X$ have real entries, i.e., if we are dealing with an operator on a real Hilbert space ${\mathcal{ H}}$ of dimension $2n$, then $U, V$ can be taken with real entries, thus are isometries on ${\mathcal{ H}}$. Do we have the same for Corollary \ref{Lin-prop1} ?  The answer might be negative, but an explicit counter-example would be desirable.

\begin{cor}\label{Lin-prop1} For every matrix in  $\bM_{2n}^+$ written in blocks of same size, we have a decomposition
\begin{equation*}
\begin{bmatrix} A &X \\
X^* &B\end{bmatrix} = U
\begin{bmatrix} \frac{A+B}{2} +{\mathrm{Im}} X &0 \\
0 &0\end{bmatrix} U^* +
V\begin{bmatrix} 0 &0 \\
0 & \frac{A+B}{2} -{\mathrm{Im}} X\end{bmatrix} V^*
\end{equation*}
for some unitaries $U,\,V\in  \bM_{2n}$.
\end{cor}

Here ${\mathrm{Im}} X=(X-X^*)/2i$. The decomposition allows to obtain some norm estimates depending on how  the full matrix is far from a block-diagonal matrix.   If $Z\in \bM_n$, its absolute value is $|Z|:=(Z^*Z)^{1/2}$. If $A, B\in \bM_n$ are Hermitian,  $A\le B$  means $B-A\in \bM_n^+$. Firstly, by noticing that ${\mathrm{Im}} X \le |{\mathrm{Im}} X |= \frac{1}{2}|X-X^*|$, we have:

\begin{cor}\label{cor-Lin-ineq} For every matrix in  $\bM_{2n}^+$ written in blocks of same size, we have
\begin{equation*}
\begin{bmatrix} A &X \\
X^* &B\end{bmatrix} \le \frac{1}{2} \left\{U
\begin{bmatrix} A+B +|X-X^*|&0 \\
0 &0\end{bmatrix} U^* +
V\begin{bmatrix} 0 &0 \\
0 & A+B+|X-X^*|\end{bmatrix} V^*\right\}
\end{equation*}
for some unitaries $U,\,V\in  \bM_{2n}$.
\end{cor}

We may then obtain estimates for the class of symmetric norms $\|\cdot\|$. Such a norm on $\bM_n$ satisfies
$\| UA\| =\| AU\| =\| A\|$ for all $A\in\bM_n$ and all unitaries $U\in\bM_n$. Since a symmetric norm on  $\bM_{n+m}$ induces a symmetric norm on $\bM_{n}$ we may assume that our norms are defined on all spaces $\bM_n$, $n\ge 1$.

\begin{cor}\label{cor-p} For every matrix in  $\bM_{2n}^+$ written in blocks of same size and for all symmetric norms, we have
\begin{equation*}
\left\|\begin{bmatrix} A &X \\
X^* &B\end{bmatrix}^p \right\| \le 2^{|p-1|}\left\{\| (A+B)^p \| +\| |X-X^*|^p \|\right\}
\end{equation*}
for all $p>0$.
\end{cor}

\begin{proof} We first show the case $0<p<1$. From  \cite{AB} (for a proof see also, \cite[Section 3]{BL}) it is known that:

\noindent
{\it If $S,T\in\bM_n^+$ and if $f:[0,\infty)\to[0,\infty)$ is concave, then, for some unitary $U,V\in\bM_n^+$,
\begin{equation}\label{AB-ineq}
f(S+T)\le Uf(S)U^* +Vf(T)V^*.
\end{equation}
}
Applying \eqref{AB-ineq} to $f(t)=t^p$ and the RHS of Corollary \ref{cor-Lin-ineq} with
\begin{equation*}
S=\frac{1}{2} U
\begin{bmatrix} A+B +|X-X^*|&0 \\
0 &0\end{bmatrix} U^*, \quad
T=\frac{1}{2}V\begin{bmatrix} 0 &0 \\
0 & A+B+|X-X^*|\end{bmatrix} V^*
\end{equation*}
 we obtain
\begin{equation*}
\left\|\begin{bmatrix} A &X \\
X^* &B\end{bmatrix}^p \right\| \le 2^{1-p}\left\{\left\| \left(A+B +  |X-X^*|\right)^p \right\|\right\}
\end{equation*}
Applying again \eqref{AB-ineq} with $f(t)=t^p$, $S=A+B$ and $T=|X-X^*|$ yields
 the   result for $0<p<1$.

To get the  inequality for $p\ge 1$, it suffices to use in the RHS of Corollary \ref{cor-Lin-ineq} the elementary inequality, for $S,T\in\bM_n^+$,
\begin{equation}\label{elem-1}
\left\| \left(\frac{S+T}{2}\right)^p \right\| \le \frac{\|S ^p\| +\|T^p \|}{2}
\end{equation}
(see  \cite[Section 2]{BL} for much stronger results). With
\begin{equation*}
S=U
\begin{bmatrix} A+B +|X-X^*|&0 \\
0 &0\end{bmatrix} U^*, \quad
T=V\begin{bmatrix} 0 &0 \\
0 & A+B+|X-X^*|\end{bmatrix} V^*
\end{equation*}
we get from Corollary \ref{cor-Lin-ineq} and \eqref{elem-1}
\begin{equation*}
\left\|\begin{bmatrix} A &X \\
X^* &B\end{bmatrix}^p \right\| \le \left\| \left(A+B +|X-X^*|\right)^p \right\|
\end{equation*}
and another application of \eqref{elem-1} with $S= 2(A+B)$ and $T= 2|X-X^*|$ completes the proof.
\end{proof}

\begin{cor}\label{cor-a} For any matrix in $\bM_{2n}^+$ written in blocks of same size such that the right upper block $X$ is accretive, we have
\begin{equation*}
\left\|\begin{bmatrix} A &X \\
X^* &B\end{bmatrix} \right\| \le \| A+B \| + \| {\mathrm{Re}} X\|
\end{equation*}
for all symmetric norms.
\end{cor}

\begin{proof}
By Corollary \ref{Lin-prop2}, for all Ky Fan $k$-norms $\|\cdot\|_{k}$, $k=1,\ldots,2n$, we have
\begin{equation*}
\left\| \begin{bmatrix} A &X \\
X^* &B\end{bmatrix} \right\|_{k} \le
\left\| \begin{bmatrix} \frac{A+B}{2} +{\mathrm{Re}} X &0 \\
0 &0\end{bmatrix} \right\|_{k}  +
\left\| \begin{bmatrix} 0 &0 \\
0 & \frac{A+B}{2} \end{bmatrix} \right\|_{k}.
\end{equation*}
Equivalently,
\begin{equation*}
\left\| \begin{bmatrix} A &X \\
X^* &B\end{bmatrix} \right\|_{k} \le
\left\|  \left(\frac{A+B}{2} +{\mathrm{Re}} X\right)^{\downarrow}  \right\|_{k}  +
\left\|
  \left(\frac{A+B}{2}\right)^{\downarrow}   \right\|_{k}
\end{equation*}
where $Z^{\downarrow}$ stands for the diagonal matrix listing the eigenvalues of $Z\in\bM_n^+$ in decreasing order.
By using the triangle inequality for $\|\cdot\|_{k}$ and the fact that
$$
\|Z_1^{\downarrow}\|_{k} +\| Z_2^{\downarrow}\|_{k} =\|Z_1^{\downarrow}+Z_2^{\downarrow}\|_{k}
$$
for all $Z_1,Z_2\in\bM_n^+$ we infer
\begin{equation*}
\left\| \begin{bmatrix} A &X \\
X^* &B\end{bmatrix} \right\|_{k} \le
\left\|  \left(A+B\right)^{\downarrow} +\left({\mathrm{Re}} X\right)^{\downarrow}  \right\|_{k}.
\end{equation*}
Hence
\begin{equation*}
\left\| \begin{bmatrix} A &X \\
X^* &B\end{bmatrix} \right\| \le
\left\|  \left(A+B\right)^{\downarrow} +\left({\mathrm{Re}} X\right)^{\downarrow}  \right\|
\end{equation*}
for all symmetric norms. The triangle inequality completes the proof
\end{proof}

\begin{cor} For any matrix in $\bM_{2n}^+$ written in blocks of same size such that $0\notin W(X)$, the numerical range the of right upper block $X$, we have
\begin{equation*}
\left\|\begin{bmatrix} A &X \\
X^* &B\end{bmatrix} \right\| \le \| A+B \| + \| X\|
\end{equation*}
for all symmetric norms.
\end{cor}

\begin{proof} The condition $0\notin W(X)$ means that $zX$ is accretive for some complex number $z$ in the unit circle. Making use of the unitary congruence
\begin{equation*}
\begin{bmatrix} A &X \\
X^* &B\end{bmatrix}
 \simeq \begin{bmatrix} A &zX \\
\overline{z}X^* &B\end{bmatrix}
\end{equation*}
we obtain the result from Corollary \ref{cor-a}.
\end{proof}

The condition $0\notin W(X)$ in the previous corollary can obviously be relaxed to  $0$ does not belong to the relative interior of $X$, denoted by $W_{int}(X)$. In case of the usual operator norm $\|\cdot\|_{\infty}$, this can be restated with the numerical radius $w(X)$:

\begin{cor} For any matrix in $\bM_{2n}^+$ written in blocks of same size such that $0\notin W_{int}(X)$, the relative interior of the numerical range the of right upper block $X$, we have
\begin{equation*}
\left\|\begin{bmatrix} A &X \\
X^* &B\end{bmatrix} \right\|_{\infty} \le \| A+B \|_{\infty} + w( X).
\end{equation*}
\end{cor}

In case of the operator norm, we also infer from Corollary \ref{Lin-prop2} the following result:

\begin{cor} For any matrix in $\bM_{2n}^+$ written in blocks of same size, we have
\begin{equation*}
\left\|\begin{bmatrix} A &X \\
X^* &B\end{bmatrix} \right\|_{\infty} \le \| A+B \|_{\infty} + 2w( X).
\end{equation*}
\end{cor}

Once again, the proof follows by replacing $X$ by $zX$ where $z$ is a scalar in the unit circle  such that $w(X)=w(zX)=\|{\mathrm{Re}}(zX)\|_{\infty}$ and then by applying Corollary \ref{Lin-prop2}.

\begin{example}\label{ex1} By letting
\begin{equation*}
A=\begin{bmatrix} 1 & 0  \\ 0&0 \end{bmatrix}, \quad B= \begin{bmatrix} 0& 0  \\ 0&1 \end{bmatrix}, \quad
X=\begin{bmatrix} 0 & 1  \\ 0&0 \end{bmatrix}
\end{equation*}
we have an equality case in the previous corollary.  This example also gives an equality case in Corollary \ref{cor-p} for the operator norm and any $ p\ge1$.  (For  any $0<p<1$ and for the trace norm, equality occurs  in Corollary \ref{cor-p} with $A=B$ and $X=0$.)
\end{example}

From Corollary \ref{Lin-prop1} we also recapture in the next corollary the majorization result obtained in \cite{LW} for positive block-matrices whose off-diagonal blocks are Hermitian.
Example \ref{ex1} shows that the Hermitian requirement on the off-diagonal blocks is necessary.

\begin{cor}\label{LW-maj} Given any matrix in $\bM_{2n}^+$ written in blocks of same size with Hermitian off-diagonal blocks, we have
\begin{equation*}
\left\|\begin{bmatrix} A &X \\
X &B\end{bmatrix} \right\|\le \| A+B \|
\end{equation*} for all symmetric norms.
\end{cor}

 Letting $X=0$ in the above corollary we get the basic inequality \eqref{class}. The last two corollaries seem to be folklore.

\begin{cor} Given any matrix in $\bM_{2n}^+$ written in blocks of same size, we have
\begin{equation*}
\left\|\begin{bmatrix} A &X \\
X^* &B\end{bmatrix} \oplus  \begin{bmatrix} A &X \\
X^* &B\end{bmatrix}  \right\|\le 2\| A\oplus B \|
\end{equation*} for all symmetric norms.
\end{cor}

\begin{proof} This follows from \eqref{class} and the obvious unitary congruence
\begin{equation*}
 \begin{bmatrix} A &X \\
X^* &B\end{bmatrix} \oplus  \begin{bmatrix} A &X \\
X^* &B\end{bmatrix}\simeq
\begin{bmatrix} A &X \\
X^* &B\end{bmatrix} \oplus  \begin{bmatrix} A &-X \\
-X^* &B\end{bmatrix}
\end{equation*}
\end{proof}

Let $||\cdot||_p$, $1\leq p < \infty$, denote the usual Schatten $p$-norms. The previous corollary entails the last one:

\begin{cor}\label{cor-schatt}  Given any matrix in $\bM_{2n}^+$ written in blocks of same size, we have
\begin{equation*}
\left\|\begin{bmatrix} A &X \\
X^* &B\end{bmatrix} \right\|_{p}\leq 2^{1-1/p} (\| A\|^{p}_{p}+\| B\|^{p}_{p})^{1/p}
\end{equation*} for all $p \in[1,\infty)$.
\end{cor}

Note that if $A=X=B$ we have an  equality case in Corollary \ref{cor-schatt}.

\begin{remark} Lemma 1.1 is still valid for compact operators on a Hilbert space, by taking $U$ and $V$ as partial isometries. A similar remark holds for the subadditivity inequality \eqref{AB-ineq}. Hence the symmetric norm inequalities in this paper may be extended to the setting of normed ideals of compact operators. 
\end{remark}

The lack of counter-example suggests that the following could hold:

\begin{conjecture} Corollary \ref{LW-maj} is still true when the off-diagonal blocks are normal.
\end{conjecture}

\vskip 10pt

J.-C. Bourin,

Laboratoire de math\'ematiques,

Universit\'e de Franche-Comt\'e,

25 000 Besancon, France.

jcbourin@univ-fcomte.fr

\vskip 10pt
Eun-Young Lee

 Department of mathematics,

Kyungpook National University,

 Daegu 702-701, Korea.

eylee89@ knu.ac.kr

\vskip 10pt
Minghua Lin

 Department of applied mathematics,

University of Waterloo,

 Waterloo, ON, N2L 3G1, Canada.

mlin87@ymail.com

\end{document}